\documentclass{article}
\usepackage[letterpaper,margin=1in]{geometry}
 \usepackage{blindtext}
 \usepackage{authblk}
 \usepackage{amsthm}
\usepackage{graphicx}
\usepackage{algorithm}
\usepackage{multirow}
\usepackage[unicode=true]{hyperref}
\usepackage[table]{xcolor}
\newtheorem{theorem}{Theorem}

\newcommand{\R}{{\mathbb R}}

\newcommand{\GL}{\operatorname{GL}}
\newcommand{\SL}{\operatorname{SL}}
\newcommand{\Trace}{\operatorname{Tr}}
\usepackage{blindtext}
\usepackage{algorithm}
\usepackage[noend]{algpseudocode}
\usepackage{pifont}
\usepackage{sidecap}
\usepackage{wrapfig}
\usepackage{lmodern}
\usepackage{enumerate}
\usepackage{enumitem}
\usepackage{comment}
\usepackage{blindtext}
\usepackage[mathscr]{euscript}

\def\be{\begin{equation}}
	\def\ee{\end{equation}}
\def\bes{\begin{eqnarray}}
	\def\ees{\end{eqnarray}}

\def\2{\frac{1}{2}}
\def\4{\frac{1}{4}}

\newcommand{\T}{\mathcal}

\newcommand{\m}[1]{{\bf{#1}}}
\newcommand{\wh}[1]{{\widehat{#1}}}

\newcommand{\mb}[1]{{\mathbb{#1}}}
\newcommand{\ms}[1]{{\mathscr{#1}}}
\newcommand{\mf}[1]{{\mathfrak{#1}}}

\usepackage{graphicx}
\newcommand{\junk}[1]{}
\usepackage[parfill]{parskip}
\usepackage{amsmath, amssymb}
\usepackage{booktabs}
\usepackage[makeroom]{cancel}
\usepackage{nicefrac}
\usepackage{listings}
\usepackage{comment}
\usepackage{float}
\usepackage{wrapfig}
\usepackage{afterpage}
\newtheorem{assumption}{Assumption}
\newtheorem{proposition}{Propositionn}
\newtheorem{definition}{Definition}

\newtheorem{lemma}[theorem]{Lemma}

\usepackage{amsmath}
\usepackage{enumitem}

\DeclareMathOperator*{\argmin}{arg\,min}
\begin{document}
% \firstpage{1}
% \subtitle{Subject Section}
\title{Coupled Tensor-Tensor Completion Method with Applications in Drug Repurposing
}
\author[1]{Maryam Bagherian}
\author[2]{Albert Hung}
\author[3,4]{Ivo Dinov}
\author[3,5]{Joshua Welch}

\affil[1]{\small Department of Mathematics \& Statistics, Idaho State University,
	Physical Science Complex | 921 S. 8th Ave., Stop 8085 | Pocatello, ID 83209}

\affil[2]{\small Computer Science \& Artificial Intelligence Laboratory, 
	Massachusetts Institute of Technology}

\affil[3]{\small Department of Computational Medicine \& Bioinformatics, 
	University of Michigan, Ann Arbor}

\affil[4]{\small Statistics Online Computational Resource (SOCR), 
	University of Michigan, Ann Arbor}

\affil[5]{\small Department of Computer Science \& Engineering, 
	University of Michigan, Ann Arbor}

\date{}

\maketitle
\noindent\textbf{Corresponding author:} Maryam Bagherian (\texttt{maryambagherian@isu.edu})

\abstract{Many biomedical challenges can be posed as tensor completion problems where the observed entries of a multidimensional array (a tensor) are used to impute the missing values. In such settings, incorporating side information about the modes of the tensor, such as gene-gene similarity, can significantly enhance the solutions of the completion problem. Most existing tensor completion methods can only incorporate side information in the form of matrices. In this study, we introduce a novel framework to incorporate side information in the form of tensors. Our new approach, called {{coupled tensor-tensor completion (CTTC)}}, leverages the hidden connections among multimodal tensors to improve tensor completion performance. In addition to practical utility, CTTC has theoretical foundations in distance metric learning and group theory. We derive an alternating algorithm to solve the CTTC optimization problem and establish its convergence to a stationary point. Finally, we show that CTTC outperforms state-of-the-art tensor completion methods at predicting drug effects. \\
	\textbf{Results:} Compared with other tensor completion methods, including HaLRTC, CTRC, Cell, and NTD-DR, CTTC demonstrates superior run-time and RSE tensor completion accuracy on two benchmark datasets, DTD and LINCS. }
\maketitle

\section{Introduction \&  Related Literature}

Many important biomedical datasets can be represented as multidimensional arrays (tensors) in which only a small subset of the entries are known. For example, the effect of a drug on a particular gene in a particular cell type may be represented as a three-dimensional array (a three-mode tensor). But most of the entries in this tensor are unobserved, because measuring the effect of many drugs on many cell types and genes is very expensive and time-consuming. 

Over the last decade, tensor completion has become an important tool with applications in many fields, including biomedical science \cite{cichocki2015tensor}. Tensor-based learning has also found broad applications beyond completion problems, including multi-view learning and clustering, where low-rank tensor representations are used to capture complementary information across multiple data sources. Tensor completion methods were developed to address the issue of missing values, \cite{bagherian2021coupled}, to name but a few. Analogous to the well-known matrix completion problem \cite{davenport2016overview}, tensor completion aims to find a tensor that identically matches the known entries of a partially observed tensor and has low-rank structure~\cite{bi2020tensors}. 

Many types of tensor completion methods have been developed. Perhaps the most popular class of tensor completion methods uses tensor decomposition to reconstruct a tensor as the product of a small number of latent factors. Numerous tensor decomposition methods have been used in this way, including Tucker \cite{tucker64extension}, Canonical Polyadic (CP) \cite{carroll1970analysis}, higher-order SVD (HOSVD) \cite{de2000multilinear}, tensor train (TT) \cite{oseledets2011tensor}, and tensor SVD (t-SVD) \cite{kilmer2013third}. Another class of tensor completion methods are based on rank regularization with {trace norm}. The minimization of matrix trace norm (or nuclear norm) is equivalent to rank minimization under certain conditions \cite{candes2012exact}. 
Thus, minimizing the tensor trace norm provides a convenient way to enforce the low-rank constraint of tensor completion.

In some applications, side information about the entries of the incomplete tensor are available. In this case, the side information can be used to further regularize the tensor completion process, which often significantly improves completion accuracy. Intuitively, when some regularity information is present, one may utilize it in order to better predict the missing entries. This leads to a new category of tensor completion methods which is based upon auxiliary information. Some of these methods fall in the following categories:

(i) Metric learning approaches utilize similarity matrices as regularization terms into the objective function. 
The recent method proposed in \cite{bagherian2021coupled} is a pertinent example where additional information is available and used for a better prediction. As another example, one may refer to \cite{narita2012tensor} where the authors proposed two regularization methods called ``within-mode regularization" and ``cross-mode regularization" to incorporate auxiliary regularity information in the tensor completion problems. The key idea is to construct within-mode or cross-mode regularity matrices, incorporate them as smooth regularizers, and then combine them with a Tucker decomposition to solve the tensor completion problem.\\

(ii) Coupled matrix/tensor factorization methods use couple tensors or matrices for jointly factorizing and imputing missing data. Among proposed methods, one may refer to coupled matrix and tensor factorization (CMTF) by \cite{acar2014structure}. The method utilizes the CP decomposition for a tensor while benefiting from the coupled information in the form of matrices.  \\

(iii) Graph-regularized tensor completion approaches incorporate graph regularization terms into either tensor completion or tensor decomposition problems. For instance, authors in \cite{takeuchi2016graph} introduced a graph Laplacian-based regularizer and used it to induce latent factors that represent auxiliary structures. Li et al. \cite{li2021imputation} used a graph-regularized tensor completion model to impute spatial transcriptomic data. The model is regularized by a product of two chain graphs that intertwines the two modes of the tensor. Many additional related papers have used graph-based side information to regularize tensor completion or decomposition \cite{sofuoglu2020graph}, or graph attention networks to enable models to focus on the most relevant information \cite{zhao2025regulation}.

(iv) Other constraint models, such as coupled trace norm\cite{wimalawarne2018convex}, couple the nuclear norm \cite{wimalawarne2018efficient} and ensemble hybrid approaches based on the prior three strategies \cite{ge2016taper}. For instance,  in addition to similarity matrices, the proposed method in \cite{ge2016taper} also incorporates the cross-mode features  as coupled matrices in order to improve the coupled tensor completion task.

However, all of these existing methods for tensor completion with side information assume that the side information comes in the form of matrices. While this is helpful in many real-world settings, there are some applications where side information is in the form of more than one matrix per tensor mode. For example, when predicting the effect of drugs on genes across cell types (a drug by cell type by gene tensor), we may have access to many different kinds of drug-drug relationships, gene-gene relationships, and cell-cell relationships. This side information can be represented as three tensors, each coupled to one mode of the drug-gene-cell type tensor (Figure 1). Here, we develop a new tensor completion approach, coupled tensor-tensor completion (CTTC), to incorporate tensor side information.
\begin{figure}
	\centering
	\includegraphics[width=2.7in]{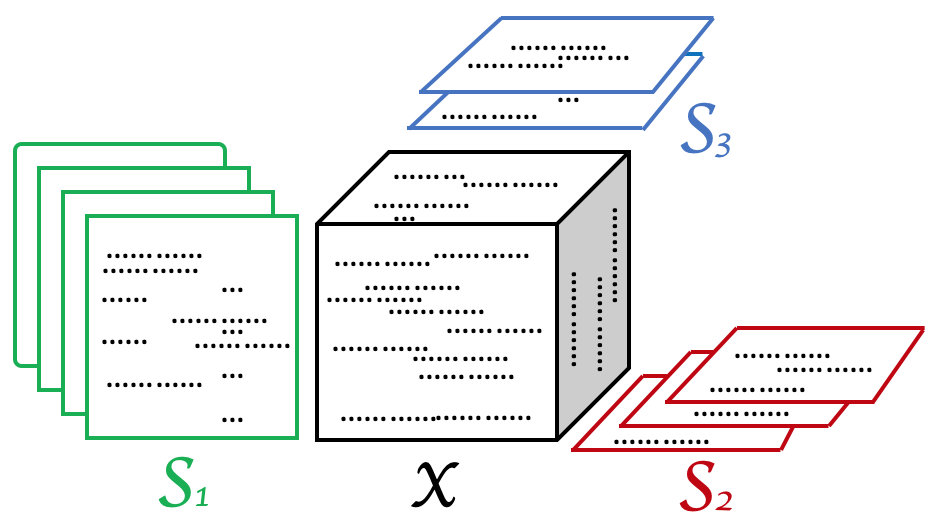}
	\caption{%
		An outline of the proposed method, CTTC. Given an unobserved/partially observed tensor $\T X$, chosen to be a three-way tensor for the sake of presentation, the proposed method aims at completing the tensor $\T X$ using similarity information provided along each mode, i.e. $\T S_1$, $\T S_2$, and $\T S_3$. As shown in the figure, the number of modes of the similarity arrays does not have to be the same as that of tensor $\T X$. The side information can be, too, under-observed and can be completed using the same method as a preprocessing step. 	}
	\label{fig:1}
\end{figure}

In order to improve the tensor completion task, while utilizing as many data relations as possible, this manuscript develops a penalized similarity based approach for coupled tensor completion in the presence of auxiliary information, also in the form of tensors. It is built upon our previous work \cite{bagherian2021coupled} for the matrix completion problem with tensor-valued side information. Building on our previously developed methods,  \cite{bagherian2022bilevel,bagherian2021coupled,bagherian2020machine,bagherian2024tensor}, we propose a penalized optimization problem for tensor completion using auxiliary information, called Coupled Tensor--Tensor Completion (CTTC), shown in Figure~\ref{fig:1}, based on a distance metric learning approach to improve the completion performance. We discuss the global convergence and demonstrate the advantages of CTTC compared to existing approaches on real datasets.

The remainder of the manuscript is organized as follows. Section~\ref{sec:BG} provides a brief background on the methodology. Section~\ref{sec:MF} introduces the CTTC problem and derives the corresponding algorithm, and the convergence analysis. The datasets and similarity tensors, are discussed in Section~\ref{sec:data} and the results are presented in Section~\ref{sec:ER}. Finally, Section~\ref{sec:conc} presents the conclusions and directions for future work.

\section{Background}\label{sec:BG}

\noindent{\textbf{Tensor Completion}}:
The goal of tensor completion is to estimate an unknown tensor ${\T{Z}} \in \mathbb R^{n_1\times n_2\times \cdots \times n_K}$ from an incomplete tensor $\T X$ where the set of indices for the observed entries is $ {\T X}_\Omega$. Throughout, we aim to find a low-rank approximation $\T Z$ that perfectly recovers the observed entries of $\T X$. Because computing the rank of $\T Z$ is intractable, rank is usually approximated as $\|\T{Z}\|_{\mu}$ with an appropriate choice of norm $\left\|\cdot\right\|_{\mu}$. Thus, tensor completion can be formulated as the following optimization problem: 
\begin{equation}
\label{eq:1}
\begin{aligned}
\min\limits_{ \T{Z} }&\quad \|\T{Z}\|_{\mu}\qquad 
\textrm{s.t.} \quad \m {\T X}_\Omega = {\T{Z}}_\Omega. 
\end{aligned}
\end{equation}

Here, we briefly introduce three main categories of tensor completion models involving non-convex optimization
problems:\\

(i)\noindent{\textbf{Tucker-Based Completion}}
One may utilize the Tucker model \cite{tucker1964extension} to perform tensor completion. Given a data tensor $\T X\in \mathbb R^{n_1\times n_2\times \cdots \times n_K}$, Tucker decomposition  approximates $\T X$ as the product of a core tensor $\T G\in \mathbb R^{\hat n_1\times \hat n_2\times \cdots\times \hat n_K}$ and one factor matrix $\m U^{(\ell)}\in \mathbb R^{\hat n_\ell\times n_\ell}$ for each tensor mode, where $\hat n_\ell \leq n_\ell$. This can be written as the following optimization problem:
\begin{equation}\label{eq:tucker}
\begin{aligned}
\min\limits_{ \Gamma_{\text t} } \quad  & \frac12\, \big{\|}\T X -\T G\times_1 \m {U}^{(1)}\times_2 \cdots \times_K \m {U}^{(K)}\big{\|}_F^2,\\
\textrm{s.t.} \quad & \m {\T X}_\Omega = {\T{Z}}_\Omega, 
\end{aligned}
\end{equation}
with $\Gamma_{\text t}:=\{\T X, \T G, \{\m{U}^{(\ell)}\}_{\ell=1}^K\}$, and $\T Z= \T G \times_1 \m U^{(1)} \times_2\cdots\times_K \m U^{(K)}$. One optimization approach is to use the block coordinate descent method to iteratively optimizing the problem. See %\cite{chen2019nonnegative,liu2023rank,zhang2016exact} for a few examples.\\

(ii)\noindent{\textbf{PARAFAC-Based Completion}}
Similarly, one may use PARAFAC (parallel factor analysis) model as:
\begin{equation}\label{eq:parafac}
\begin{aligned}
\min\limits_{ \Gamma_{\text p} } \quad  & \frac12\, \big{\|}\T X - \m {U}^{(1)}\circ \cdots \circ \m {U}^{(K)}\big{\|}_F^2, \\
\textrm{s.t.} \quad & \m {\T X}_\Omega = {\T{Z}}_\Omega, 
\end{aligned}
\end{equation}
where $\Gamma_{\text p}:=\{\T X,  \{\m{U}^{(\ell)}\}_{\ell=1}^K\}$, $\circ$ denotes the outer product and $\m U^{(\ell)}\in \mathbb R^{n_\ell \times R}$, for each $\ell=1, \cdots, K$ and for some positive $R\in \R$ as the number of component.
Note that Eq.~\eqref{eq:tucker} may be regarded as a more flexible PARAFAC model. In PARAFAC the core tensor is restricted to be \textit{super diagonal} \cite{kolda2009tensor}.
% See \cite{ashraphijuo2017fundamental,liu2019low} for a few example.\\

(iii)\noindent{\textbf{Matrix-Based Models}}
One may reshape the tensor as
multiple matrices and force the low-rank restriction on each unfolding matrix along each
mode of the tensor. Given a tensor $\T X\in \mathbb R^{n_1\times n_2\times \cdots \times n_K}$, let $\m X_{(\ell)}$ denote the $\ell$-mode matricization of tensor $\T X$, therefore, one may solve the following problem: 
\begin{equation}\label{eq:matrix}
\begin{aligned}
\min\limits_{ \Gamma_{\text m} } \quad  & \frac12\,\sum\limits_{\ell=1}^K \big{\|}\m X_{(\ell)}- \m M_\ell\big{\|}_F^2, \\
\textrm{s.t.} \quad &{\T X}_\Omega =  {\T{Z}}_\Omega, \quad \text{and} \quad \text{rank}(\m M_\ell) \le R, 
\end{aligned}
\end{equation}
for $\ell=1, \cdots, K$. Here, where $\Gamma_{\text m}:=\{\T X,  \{\m{M}_{\ell}\}_{\ell=1}^K\}$, and $\m M_\ell\in \R ^{n_\ell \times J}$, where $J=\prod_{k\neq \ell}^K n_k$.
% (See \cite{xu2013parallel,ji2016tensor} for a few examples).\\

\noindent{\textbf{Coupled Tensor Completion}}:
The goal of coupled tensor completion is to estimate an unknown tensor ${\T{Z}} \in \mathbb R^{n_1\times n_2\times \cdots \times n_K}$ from an undersampled/incomplete tensor $\T X$ where the set of indices for the observed entries is $\m \Omega_{\T X}$. Moreover, for each mode of tensor $\T X$, say $\ell$, where $\ell=1, 2, \cdots, K$, some similarity/auxiliary information in the form of a matrix, $\m S_\ell$ or a tensor, $\T S_\ell$ are coupled, similar to \cite{bagherian2021coupled, bagherian2022bilevel} where in the former the goal is to perform matrix completion using coupled information whereas in the latter a tensor is being recovered by using the existing side information or creating them.

\subsection{Distance Metric Learning}
The proposed method (see Section \ref{sec:cttc}) is constraints by Distance metric learning terms in order to effectively learn the similarities, as the inverse of distance, between any two interies. Distance metric learning \cite{suarez2018tutorial} aims to learn distances from the data, where distance refers to a map $d:A\times A\to \mb R_+$, where $A$ is a non-empty set, then for all $\m a, \m b, \m c\in A$, $d$ satisfies the following conditions: coincidence,  Symmetry, and triangle inequality.

Other properties such as \emph{non-negativity}, \emph{reverse triangle inequality}, and \emph{generalized triangle inequality} follow immediately from the definition above. 
Distance metric learning frequently focuses on learning Mahalanobis distances, since they
are parametrized by matrices, and therefore are computationally tractable. Mahalanobis distances satisfy additional properties, including \emph{translation invariance} and \emph{homogeneousness} \cite{suarez2018tutorial}. In a $K$-dimensional Euclidean space, we may form a family of metrics over the set $A$ by computing Euclidean distances after a linear transformation $\m a\to\m L(\m a)$, where $\m L$ is injective. Therefore, the squared distances can then be computed as $d(\m a,\m b)=\|\m L(\m a-\m b)\|^2_F$ which may also be expressed in terms of a positive semidefinite square matrix (i.e. $\forall \mathbf x\in X, \mathbf x^\top \mathbf M x\ge 0$) and $\m M = \mathbf L\mathbf L^\top$. This refers to as single-metric learning which learns a metric tensor, here $\mathbf M$, such that distances are measured as
$$\mathrm{dist}(\mathbf x_i,\mathbf x_j)=\|\mathbf x_j-\mathbf x_j\|^2_M\equiv (x_i-x_j)^\top \mathbf M (x_i-x_j). $$

If $\m L$ is also surjective, which results in $\m M$ being full rank, the matrix $\m M$ parametrizes the distance $d$. An example of matrix learning is when the matrix $\m M$ is referred to as a \emph{Mahalanobis} metric. In Gaussian distributions, the matrix $\m M$ plays the role of the inverse covariance matrix. The following definition gives the generalization of the Mahalanobis distance to multi-linear transformation.
\begin{definition}\label{def:maha}
Let $\T X\in \mathbb R^{n_1\times n_2\times \cdots\times n_K}$. Consider the multilinear transformation \\
$ \varphi:\mathbb R^{n_1\times n_2\times \cdots\times n_K}\to \mathbb R^{n_1\times n_2\times \cdots\times n_K}$, with
\begin{align}
\varphi(\T X)= \T X\times_1 \m L^{(1)}\times_2 \m L^{(2)}\cdots \times_K \m L^{(K)}, 
\end{align}
where the square matrices $\m L^{(\ell)}\in \mathbb R^{n_\ell\times n_\ell}$, for $\ell=1,\cdots, K$, are called the $\ell$-mode matrices. Here, if $\m L^{(\ell)}$, for $\ell=1, \cdots, K$ are orthogonal matrices, then $ \varphi$ recovers Euclidean distance.
\end{definition}

	The coupling term in CTTC is constructed from the tensor-valued Mahalanobis distance. Therefore, it is important to establish that this metric is stable with respect to perturbations in the auxiliary tensors. The following proposition provides a perturbation estimate showing that the induced metric varies continuously under small perturbations.

\begin{proposition}[Perturbation bound for the tensor Mahalanobis distance]
	\label{prop:perturb}
	Let $\T X,\T S,\T E\in
	\mathbb R^{n_1\times\cdots\times n_K}$, and let
	\[
	d_M(\T X,\T S)
	=
	\Big\|
	(\T X-\T S)
	\times_1\m L^{(1)}
	\cdots
	\times_K\m L^{(K)}
	\Big\|_F^2
	\]
	be the tensor Mahalanobis distance defined in
	Definition~\ref{def:maha}. Define
	\[
	C=\prod_{\ell=1}^K\|\m L^{(\ell)}\|_2.
	\]
	Then, for any perturbation $\T E$,
	\[
	\begin{aligned}
		&
		\left|
		d_M(\T X,\T S+\T E)
		-
		d_M(\T X,\T S)
		\right|
		\\
		&\le
		2C
		\sqrt{d_M(\T X,\T S)}
		\,\|\T E\|_F
		+
		C^2
		\|\T E\|_F^2.
	\end{aligned}
	\]
Consequently, on bounded subsets of the tensor space, the tensor Mahalanobis distance is locally Lipschitz continuous with respect to perturbations of the auxiliary tensor.
\end{proposition}
\begin{proof}
	Let
	\[
	\T A=
	(\T X-\T S)
	\times_1\m L^{(1)}
	\cdots
	\times_K\m L^{(K)},
	\]
	and
	\[
	\T B=
	\T E
	\times_1\m L^{(1)}
	\cdots
	\times_K\m L^{(K)}.
	\]
	Then
	\[
	d_M(\T X,\T S+\T E)
	=
	\|\T A-\T B\|_F^2,
	\]
	while
	\[
	\|\T A\|_F
	=
	\sqrt{d_M(\T X,\T S)}.
	\]
	
	Using the identity
	\[
	\bigl|
	\|\T A-\T B\|_F^2-\|\T A\|_F^2
	\bigr|
	\le
	2\|\T A\|_F\|\T B\|_F+\|\T B\|_F^2,
	\]
	together with the multilinear norm inequality
	\[
	\|\T B\|_F
	\le
	\prod_{\ell=1}^K
	\|\m L^{(\ell)}\|_2
	\,
	\|\T E\|_F
	=
	C\|\T E\|_F,
	\]
	we obtain
	\[
	\begin{aligned}
		\left|
		d_M(\T X,\T S+\T E)
		-
		d_M(\T X,\T S)
		\right|
		&=
		\bigl|
		\|\T A-\T B\|_F^2-\|\T A\|_F^2
		\bigr|\\
		&\le
		2\|\T A\|_F\|\T B\|_F+\|\T B\|_F^2\\
		&\le
		2C\sqrt{d_M(\T X,\T S)}\,\|\T E\|_F
		+
		C^2\|\T E\|_F^2,
	\end{aligned}
	\]
	which proves the desired perturbation estimate.
	\end{proof}

	Proposition~\ref{prop:perturb} shows that small perturbations in the auxiliary tensors produce proportionally small perturbations in the induced tensor Mahalanobis distance. This stability property provides theoretical justification for the use of tensor-valued side information within the CTTC framework.

Our work is inspired by relation regularized matrix factorization %\cite{li2009relation,narita2012tensor,song2019tensor} 
where two regularization methods called ``within-mode regularization'' and ``cross-mode regularization'' are used to incorporate auxiliary similarity among modes for tensor factorization. These
methods are all based on an expectation maximization (EM)-like approach combined with Tucker or PARAFAC decomposition. The key
idea is to construct within-mode or cross-mode similarity matrices and incorporate them as regularization terms:
\begin{equation}
  \label{eq:with}
    R_{\text{within}}(\T X; \m U^{(1)}, \cdots, \m U^{(N)})= \sum\limits_{n=1}^N \sum\limits_{i,j=1}^{I_n} \m S_n(i,j) \times \| \m U^{(n)}(i, :)- \m U^{(n)}(j, :)\|, 
\end{equation}
which can be simplified to $R_{\text{within}}(\T X; \m U^{(1)}, \cdots, \m U^{(N)})= \sum\limits_{n=1}^N\Trace(\m U^{(n)^\top} \m L_n \m U^{(n)}),$
and
\begin{equation}
    \label{ed:cross}
 R_{\text{cross}}(\T X; \m U^{(1)}, \cdots, \m U^{(N)})= \Trace\bigg(
    (\m U^{(1)^\top}\otimes \cdots\otimes\m U^{(N)^\top}) 
    \m L (\m U^{(1)}\otimes \cdots \otimes \m U^{(N)})\bigg ),
\end{equation}
where $\m U^{(n)}$ is the mode-$n$ latent matrix of PARAFAC or Tucker decomposition, $\m S_n$ is the mode-$n$ auxiliary
similarity matrix of size $I_n \times I_n$, $\mathbf L_n$ is the Laplacian matrix of $\mathbf S_n$ and $\mathbf L$ is defined as the Laplacian
matrix of $\mathbf S_1 \otimes \cdots \otimes \mathbf S_N$ , which is also called \emph{Kronecker product similarity} in the early work of \cite{kashima2009link}. In summary, DML in tensor completion is used to learn a meaningful similarity measure between data points, enabling more accurate predictions and reconstructions of missing entries in incomplete tensors. By optimizing the distance metrics, it helps in capturing the inherent relationships between elements, improving the completion process, particularly in high-dimensional data such as multi-way arrays. This approach enhances the ability to predict missing values by leveraging the structure of the data and the learned metrics to guide the completion algorithm effectively. 
	The representation-theoretic formulation underlying the proposed metric-learning framework, including the associated reductive group action and its relation to the Real Kempf--Ness theorem, is developed in detail in \cite{bagherian2026algebraic} and is not repeated here.

\section{CTTC Method}\label{sec:MF}

\label{sec:cttc} 
This method is an adoption of the method introduced in \cite{bagherian2026algebraic}. The proposed method utilizes Distance metric learning in order to effectively learn the similarities, as the inverse of distance, between any two entries. Distance metric learning \cite{suarez2018tutorial} aims to learn distances from the data, where distance refers to a map $d:A\times A\to \mb R_+$, where $A$ is a non-empty set, then for all $\m a, \m b, \m c\in A$, $d$ satisfies the following conditions: coincidence,  Symmetry, and triangle inequality.
Other properties such as \emph{non-negativity}, \emph{reverse triangle inequality}, and \emph{generalized triangle inequality} follow immediately from the definition above. 
Distance metric learning frequently focuses on learning Mahalanobis distances, since they
are parametrized by matrices, and therefore are computationally tractable. Mahalanobis distances satisfy additional properties, including \emph{translation invariance} and \emph{homogeneousness} \cite{suarez2018tutorial}.
Let the data tensor $\T X\in  \mb R^{n_1\times n_2\times n_3}$ be given. For any fixed mode, we assume that
the following symmetric tensors, in form of auxiliary information, are given:
\begin{equation}
    \begin{cases}
    \T S^{(1)}\in S^2X_{n_1}\otimes V_1, \\
    \T S^{(2)}\in S^2X_{n_2}\otimes V_2,\\
    \T S^{(3)}\in S^2X_{n_3}\otimes V_3,
    \end{cases}
\end{equation} 
such that they are coupled with an incomplete tensor $\T X$. Here, for $\ell=1,2,3,$ $S^2X_{n_\ell}$'s denote the symmetric power of $X_{n_\ell}$ and $V_\ell$'s are some arbitrary vector spaces over $\mb R$. This allows the method to utilize any size of arrays to represent available auxiliary information.
Without loss of generality, one may assume that $\T S^{(1)}$, $\T S^{(2)}$, and $\T S^{(3)}$ are nonnegative definite. For each fixed $\ell=1,2,3$, the $j$th mode matricization of tensor $\T S^{(\ell)}$ is a symmetric matrix denoted by $\m S^{(\ell)}_{(j)}$ .
The space of symmetric $n_\ell\times n_\ell$ matrices can be identified with the space
of symmetric tensors $S^2X_{n_\ell}\subseteq X_{n_\ell}\otimes X_{n_\ell}$.

Let $\GL(\cdot)$ and $\SL(\cdot)$ denote the general linear group and special linear group, respectively. The symmetric group $\GL(X_{n_\ell})$ acts on a symmetric matrix $\m B$ by the action defined as:
$$\m A\cdot \m B\to \m A\m B\m A^\top,$$
for $\m A\in \GL(X_{n_\ell}) $ and $\m B \in X_{n_\ell}\otimes X_{n_\ell}$.
Suppose that the only known entries of $\T X$ are at positions 
$$\Omega_{n_1n_2n_3}=((i_1,j_1,k_1),(i_2,j_2,k_2),\dots,(i_m,j_m,k_m)).$$
This 
constraint can be written as $\omega_{n_1n_2n_3}(\T Z)=\nu_{m}$ where $\nu_m\in \R^m$ is some fixed vector, and 
$$\omega_{n_1n_2n_3}:\ms X\to \R^m,$$ maps a tensor $\T B\in \ms X$ to
the vector \\$(\T B_{i_1,j_1,k_1},\T B_{i_2,j_2,k_2},\dots,\T B_{i_m,j_m,k_m})^\top\in \R^m$.
One may use the similarity tensors $\T S^{(1)}$, $\T S^{(2)}$ and $\T S^{(3)}$ as regularization of the tensor completion problem of tensor $\T X$. 

\subsection{CTTC Optimization Problem}\label{subsec:cttco}
Let $\T X\in \mb R^{n_1\times n_2\times n_3}$ be the data tensor given. For any fixed mode, say $\ell$, $\ell=1,2,3$, there exists a similarity tensor denoted by $\T S^{(\ell)}$. The goal is to find the unknown tensor $\T Z$, that matches the known entries of tensor $\T X$; The objective function to minimize is therefore:
\begin{equation}\label{eq:main}
\begin{aligned}
\min\limits_{\Gamma}  & \, 
F(\T Z):=\frac12\big{\|}\T Z \times_1 \m U^{(1)}\times_2 \m U^{(2)}\times_3   \m U^{(3)}\big{\|}_F^2
% \\\quad  &\qquad\quad
+\sum\limits_{\ell=1}^3 J_\ell(\m U^{(\ell)})\\ \quad  &
\qquad\quad+\sum\limits_{\ell=1}^3 \varrho_\ell\,  \sum\limits_{j=1}^{\text{num}_\ell}\,  (\m D_\ell)_{j,j}\,
\Trace\left (\m U^{(\ell)}\, \m S^{(\ell)}_{(j)}\, \m U^{{(\ell)^\top}}\right ), \\
\textrm{s.t.}  & \, \omega_{n_1n_2n_3}(\T Z)=\nu_{m}, 
\end{aligned}
\end{equation}
where $\Gamma=\{\T Z,\m U^{(1)}, \m U^{(2)}, \m U^{(3)}\}$, with $\m U^{(1)}\in \SL_{n_1}$, $\m U^{(2)}\in \SL_{n_2}$, $\m U^{(3)}\in \SL_{n_3}$, and $\T Z\in \R^{n_1\times n_2\times n_3}$, for some regularization parameters $\varrho_\ell$, and penalty functions $J_\ell(\cdot)$, for $\ell=1,2,3$. Here, $\m S^{(\ell)}_{(j)}$ is the $j$-th mode matricization of tensor $\T S^{(\ell)}$. Moreover, the objective function given in Problem~\eqref{eq:main} also minimizes all the number of layers added to $\T S^{(\ell)}$, denoted as $\text{num}_\ell$, for each layer.  Additional transformations $\m D_\ell$ for $\ell=1,2,3$, which are diagonal matrices with unit determinant and positive entries on the diagonal, are also adopted. If the auxiliary tensors $\T S^{(\ell)}$ have missing elements, certain constraints are adopted in addition to the one which assumes that all 
the slices are nonnegative definite. The matrices $\m U^{(i)}$ here define the Mahalanobis distance as given in Definition~\ref{def:maha} as a form of DML. Equivalently, $F(\T Z)$ in Problem~\ref{eq:main} can be written as:  
\begin{equation}\label{eq:main:alt}
	F(\T Z):=\frac12\sum\limits_{\ell=1}^3\big{\|}\m U^{(\ell)}\m Z_{(\ell)} \bigotimes\limits_{t\neq \ell}^3\m U^{(t)^\top}\big{\|}_F^2.
\end{equation}

The optimization algorithm follows an alternating minimization scheme. Each of the matrices, $\m U^{(\ell)}$, $\ell=1,2,3$, is being optimized while the other two are held fixed. As for the penalty functions over each of these matrices, we have imposed the trace norm as a tight convex approximation for the rank of a matrix. It is shown, theoretically, that for matrices and under certain conditions, rank minimization problem can be translated to solving the convex optimization problem of minimization of the trace norm over the given
affine space \cite{candes2012exact}. One may generalize the definition of the trace norm for matrices to tensors as a convex combination of
the trace norms of all matricization along each mode. Therefore, in essence, the trace norm penalty function can also be imposed over tensor $\T Z$.

An schematic overveiw of the model is provided in Figure~\ref{fig:model}. 

\begin{figure}
	\centering
	\includegraphics[width=0.8\linewidth]{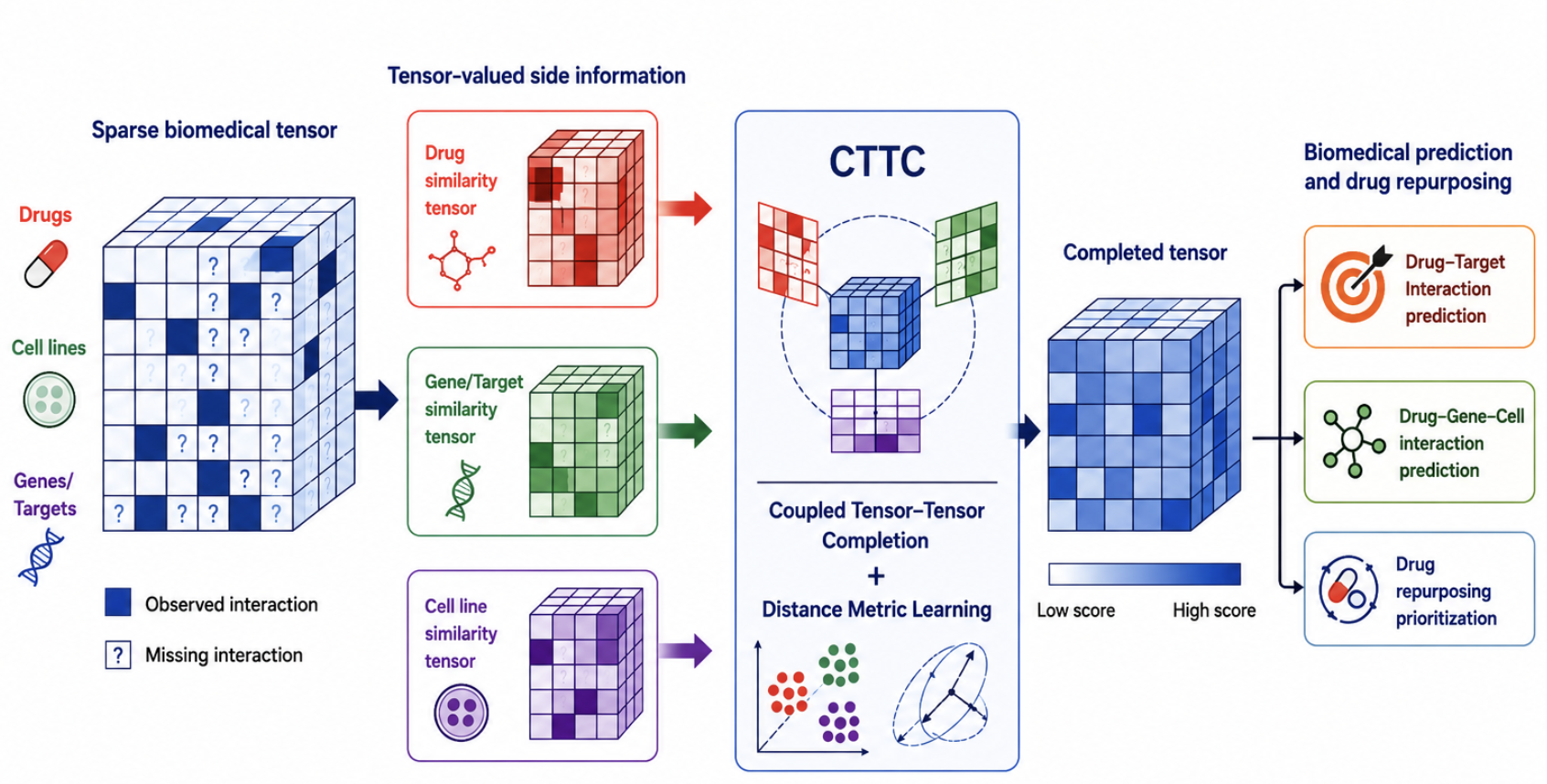}
	\caption{%
		Overview of the proposed Coupled Tensor--Tensor Completion (CTTC) framework. CTTC integrates a sparse primary biomedical tensor with tensor-valued side information representing drug, gene/target, and cell-line similarities. The framework jointly performs tensor completion and distance metric learning to recover missing interactions and generate predictions for downstream biomedical applications, including drug--target interaction prediction, drug--gene--cell interaction prediction, and drug-repurposing prioritization.	}
	\label{fig:model}
\end{figure}

\subsection{Convergence Analysis}\label{sec:con.a}
In this section, we outline the conditions under which Problem~\eqref{eq:main} converges to a stationary point. For any proper, lower semi-continuous function $\mf f : \mb E \rightarrow (-\infty, \infty]$, we let $\partial_L \mf f : \mb E \rightarrow 2^{\mb E}$ denote the \textit{limiting subdifferential} of $\mf f$; see~\cite[Definition 8.3]{rockafellar2009variational}.
We consider the following minimization problem: 
\begin{equation}\label{eq:ca:1}
\min\limits_{\Gamma} \big{\{}F(\T Z)\equiv f(\T Z)+h(\T Z)+g_1(\m U^{(1)})+ g_2(\m U^{(2)})+g_3(\m U^{(3)})\big{\}}, 
\end{equation}
where:
\begin{enumerate}
    \item The functions $\mf f : \mb E \to (-\infty, \infty]$, where $\mf f=f,h,g_\ell$, $\ell=1,2,3$, are proper semi-continuous and subdifferentiable over their domain, 
\item The function $f$ is a continuously differentiable convex function over $\text{Dom}(h) \times\text{Dom}(g_1) \times \text{Dom}(g_2) \times\text{Dom}(g_3)$. 
\end{enumerate}
We consider the following assumption: 
\begin{assumption}\label{assump:A}
Considering Problem~\eqref{eq:main}:
\begin{enumerate}
\item $J: \R^{n_{1} \times n_{2} \times n_{3} } \rightarrow \left(-\infty , \infty\right]$  are proper and lower semi-continuous such that\\ $\inf_{\R^{n_{1} \times n_{2} \times n_{3} }}J > -\infty$, 
    \item $F(\T Z): \R^{n_{1} \times n_{2} \times  n_{3}} \rightarrow \R$ is differentiable and $\inf_{\R^{n_{1} \times n_{2} \times n_{3}}} F > -\infty$.
    \item  The gradients $\nabla F(\T Z)$ is Lipschitz continuous with moduli $L_F$, i.e.,
\begin{equation*}
\|\nabla F(\T Z^1) - \nabla F(\T Z^2)\|_F^2 \leq L_F\| \T Z^1 - \T Z^2 \|_F^2, \quad \forall \quad \T Z^{1}, \T Z^{2}.
  \end{equation*} 
\end{enumerate}
\end{assumption}
Here, we adopt an alternating optimizer to solve the optimization problem. The ability to employ the alternating minimization method relies on the capability of computing minimizers with respect to each of the blocks.  Here, we use the results of Bertsekas \cite{bertsekas1999nonlinear}  stating that if the minimum with
respect to each block of variables is \emph{unique}, then any accumulation point of the sequence
generated by the method is also a stationary point. In order to establish the convergence results for Problem~\eqref{eq:main}, we first show that the problem is indeed well-defined, in the most general sense. For any $\ell=1,2,\cdots, K$, we consider the Euclidean space $\mb E=\mb E_1\times\mb E_2\times \cdots\times \mb E_K,$ and the set of linear transformation $\{f_\ell\}_{\ell=1}^K$ , where $f_\ell: \mb E_\ell\to \mb E$, mapping all but $\ell$-th block to zero. We now have the following lemma:
\begin{lemma}[Well-definedness]\label{lem:wd}
Suppose $F:\mb E\to (-\infty, \infty]$ is a proper and closed function. Further, we assume
that $F$ has bounded level sets; i.e. the set $\textup{Level}(F,\mu):=\{\T Z\in \mb E: F(\T Z)\le \mu\}$ is bounded for any $\mu \in \mb R$. Then the function $F$ has at least one minimizer, and for any $\bar{\T Z}\in \text{Dom}(F)$, the problem:
\begin{equation}
\label{eq:lem:wd}
\begin{aligned}
\min\limits_{ {\m Y}_\ell\in \mb E_\ell }&\quad F\left(\bar{\T Z}+f_\ell({\m Y}_\ell-\bar{\m Z}_\ell)\right),
\end{aligned}
\end{equation}
for $\ell=1,2,\cdots, K$, possesses a minimizer.
\end{lemma}

\begin{proof}
For $\wh{\T Z}\in \text{Dom}(F)$, 
\begin{equation}
    \argmin\limits_{\T Z\in \mb E} F(\T Z)= \argmin\limits_{\T Z\in \mb E} \big{\{}F(\T Z):\T Z\in \textup{Level}(F, F(\wh{\T Z})\big{\}}.
\end{equation}
Since $F$ is closed with bounded level sets, then $\textup{Level}(F, F(\wh{\T Z})$ is compact. It follows from Weierstrass theorem that the problem of minimizing $F$ over $\textup{Level}(F, F(\wh{\T Z})$, and therefore over the entire space possesses a minimizer. The same argument applied to the function defined in Eq.~\eqref{eq:lem:wd}, as a proper and closed function with
bounded level sets.
\end{proof}

\begin{theorem}[Stationarity of accumulation points]\label{thm:con}
	Suppose Assumption~\ref{assump:A} holds, the sequence generated by
	the model is bounded, and the minimum of each block subproblem
	is uniquely attained. Then every accumulation point of the sequence
	\[
	\Gamma_k=
	\bigl(\T Z_k,\m U_k^{(1)},\m U_k^{(2)},\m U_k^{(3)}\bigr)
	\]
	is a stationary point of Problem~\eqref{eq:main}.
\end{theorem}

 \begin{proof}
 The Problem~\eqref{eq:main} is well-defined and the sequence $\{\Gamma_k\}$ is non-empty and bounded based on Lemma~\eqref{lem:wd}. Proof follows from Proposition~{[2.7.1]} in \cite{bertsekas1999nonlinear} (also see Theorem~{[14.3]} in \cite{beck2017first}). Based on Proposition~{[2.7.1]}, if a minimum of each subproblem of an alternating minimizer is uniquely attained, then any limit point of the sequence $\{\Gamma_k\}$ is a stationary point.
\end{proof}

Lastly, it is noteworthy that the only result available on the \emph{rate of convergence}
of the alternating method under general convexity assumptions (and not strong convexity) is the result
in \cite{beck2013convergence} showing a sublinear rate of convergence, and that the multiplicative constant depends
on the minimum of the block Lipschitz constants. However, this result is limited in
the sense that it only holds for unconstrained problems with a smooth objective function, and therefore does not apply to our constrained problem. 
%%%%%%%%%%%%%%%%%%%%%%%%%%

%%%%%%%%%%%%%%%%%%%%%%%%
\section{Datasets and Similarity Tensors}\label{sec:data}

\subsection{Datasets}
We demonstrate the real-world utility of CTTC using two datasets, both related to drug repurposing. The idea of drug repurposing is to learn from a small number of contexts in which drugs have been previously observed, then use tensor completion to predict how drugs will perform in new contexts. For example, we could use tensor completion to identify new cell lines in which a drug could prove effective based on similarity between the new cell line and previous cell lines that we have profiled. This allows drug repurposing, in which the same drug can be applied in new contexts. Lastly, any additive noise can be effectivly removed using the duality of the spectral and nuclear Schatten norms of tensor as proposed in \cite{bagherian2024tensor}. To this end, we investigate the following datasets.\\
\noindent\textbf{\noindent {DTD}}: The Drug-Target-Disease (DTD), provided in \cite{jamali2022ntd}, combined public data from multiple sources to construct a drug-target association matrix. The data were downloaded from DrugBank \cite{wishart2006drugbank}, UniProt \cite{uniprot2018uniprot}, and SuperTarget \cite{gunther2007supertarget}. There are three association matrices which ultimately form the main three-way tensor, drug-target-disease, with $\T X\in \mb R^{I\times J\times K}$
%with $810\times 302\times 542$ 
$114,319$ triplet associations. The drug-target association matrix, $\m A_{\text{CT}}\in \mathbb R^{I\times J}$, consists of $13,898$ associations. The authors downloaded drug-disease associations from Online Mendelian Inheritance in Man (OMIM) \cite{amberger2019omim} and the Comparative Toxicogenomics Database (CTD) \cite{davis2021comparative}, and $\m A_{\text{CD}}\in \mathbb R^{I\times K}$ is constructed with $550319$ known associations. Moreover, $14730$ associations between targets and diseases are retrieved from the Comparative Toxicogenomics Database, OMIM, Uniprot, DisGeNET \cite{pinero2020disgenet}, and GAD \cite{becker2004genetic} to construct $\m A_{\text{TD}}\in \mathbb R^{J\times K}$. 
%\item \noindent\textbf{Data} 
The similarity matrices along each mode have been extended to similarity tensors whose third modes are formed by averaging over the values of first and second modes. The final tensor is of size  $810\times 302\times 542$, with modes corresponding to chemicals, targets, and diseases (see Figure.~\ref{fig:data}). Each element of the tensor is a drug that targets a gene in a disease.\\
\noindent\textbf{LINCS}: Library of Integrated Network-based Cellular Signatures \footnote{\url{https://lincs.hms.harvard.edu/db/}} is a dataset of gene expression levels across multiple cell lines after a variety of drug treatments. Specifically, the target tensor is created by using the level 5 data from the Broad Institute LINCS Phase 1 L1000 dataset. The drug similarity tensor is created by using two of the drug matrices from the previous method. An additional layer of similarity information is obtained from ChemicalChecker \cite{duran2020extending}. A dictionary maps the LINCS drug information to their corresponding DrugBank \cite{wishart2006drugbank} counterpart, ensuring that each index is aligned. The gene similarity tensor is obtained by isolating the intersection of the genes present in all of the datasets. Ultimately, there were five sources of information for the gene similarity tensor, including two similarity matrices from the previous data, one matrix from the NTD-DR data \cite{jamali2022ntd}, one matrix from STRING \cite{szklarczyk2021string}, and one matrix from the Regnetwork \cite{liu2015regnetwork}. Lastly, the cell line similarity information has been obtained from the attribute similarity matrix provided in \href{https://maayanlab.cloud/Harmonizome/dataset/LINCS+Kinativ+Kinase+Inhibitor+Bioactivity+Profiles}{LINCS Kinativ Kinase Inhibitor Bioactivity Profiles}. The final tensor is of size $6780\times 70\times 1104$, where the modes are drugs, cell lines, and genes (see Figure.~\ref{fig:data} (left)). Each element of the tensor indicates the expression level of a gene in a cell line treated with a drug.\\

 \subsection{Tensor Similarity Calculations} The similarity information between vector entries $x$ and $y$ is calculated by $S(x, y) = \frac{1}{1 + d(x, y)}$, where $d$ is a suitable distance function, e.g. the Euclidean distance. The vector entries $x$ and $y$ are obtained from the provided databases detailed below. For the full adjacency tensors, the entries were min-max normalized of $[0, 1]$ and the missing values were imputed with $-1$.

$25$ layers of chemical similarity information was calculated from \textit{Chemical Checker}, encompassing chemical properties (2D and 3D fingerprints, scaffold, structure, and physiochemical), chemical targets (mechanisms of action, metabolic genes, crystals, bindings, HTS bioassays), networks (small molecule roles and pathways, signaling pathways, biological processes, and interactome), cell information (impact on gene expression, cancer cell lines, chemical genetics, morphology, and cell bioassasy), and clinical information (therapeutic assays, indications, side effects, disease and toxicology, and drug-drug interactions). 
Drug-drug similarities, similar to \cite{bagherian2021coupled}, are calculated from DrugBank, which contains drug target information, from the Morgan Fingerprint score, \cite{rogers2010extended} provided by the RDKit Python package \cite{rdkit}, which contains structural information, and the topological torsion score, \cite{nilakantan1987topological} also provided by the RDKit Python package, which represents structural information. For more details we refer the reader to \cite{bagherian2021coupled}.

The cell-cell similarities represent the percent inhibition of Kinase by small molecules in different cell lines (from the LINCS dataset), and the cell line ontology using \textit{Hamming distance}, as measure of the difference between two strings of equal length. For gene-gene interactions, on the other hand, the STRING database \cite{szklarczyk2019string} incorporated both physical and functional information of proteins. Regnetwork \cite{liu2015regnetwork} incorporates genome wide regulatory network information by taking into account transcription factors, microRNAs and target genes. Binary target-target information from BioGrid$^{3.5}$ \cite{stark2006biogrid} and the target target similarity using the inverse \textit{Jukes-Cantor} distance \cite{jukes1969evolution}, which contains sequence similarity information are similar to those of \cite{bagherian2021coupled}. From NTD-DR, the final fused tensor was used after calculating sequence based similarity, protein protein interactions and gene ontology information.  
\begin{figure}
    \centering
    \includegraphics[width=2.7in ]{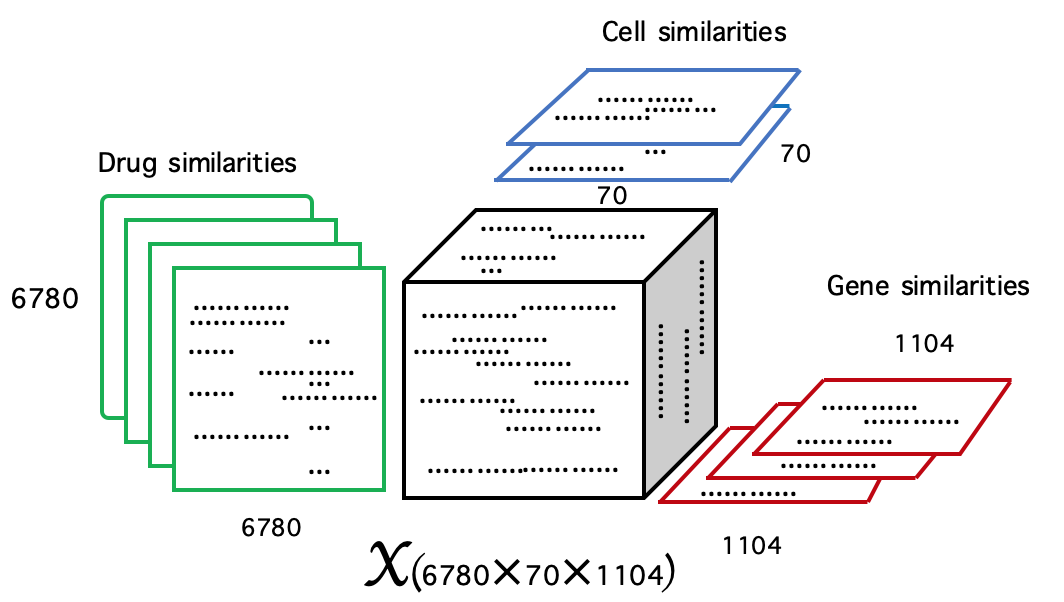}
    \includegraphics[width=2.7in]{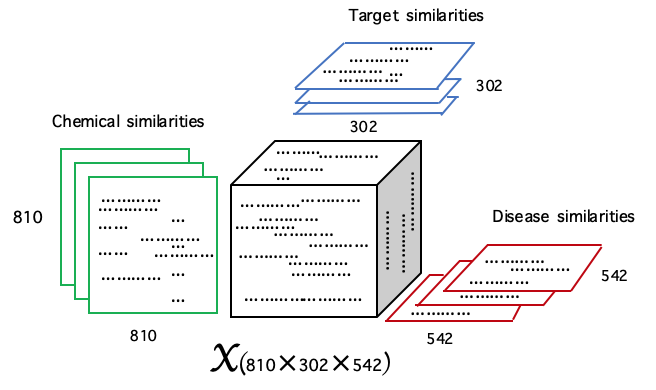}
    \caption{An overview of the LINCS dataset (top) and DTD dataset (bottom) shown with the layers of coupled information. }
    \label{fig:data}
\end{figure}

\subsection{Other state-of-the-art methods}
We compared our proposed method with two classes of tensor completion methods: one that does not utilize any side information, yet is developed for specific datasets such as DTD or LINCS, and one that benefits from incorporating such information, if it exists, but is limited in the amount of information that can be coupled or in the number of modalities of the coupled array, which, in this case, is limited to single matrices.\\
\noindent\textbf{NTD-DR: Non-Negative Tensor Decomposition for Drug Repositioning \cite{jamali2022ntd}} is a graph-regularized-based method which utilizes pairwise associations to construct a three-dimensional tensor whose modes represents  drugs, targets, and diseases. It creates Laplacian matrices as similarity information.\\
\noindent\textbf{{CTRC} \cite{huang2020unified}}: Coupled Tensor Ring Completion is a TR-decomposition based method for coupled completion. It utilizes a block coordinate descent algorithm to optimize the solution.   \\
\noindent\textbf{Cell-specific prediction and application of drug-induced gene expression profiles  \cite{hodos2018cell}}
This method utilizes rather classical approaches for the specific application discussed in the paper. It uses either well-known $k$- nearest neighbor for local imputation or Fast Low Rank Tensor Completion (FaLRTC) method \cite{liu2012tensor}, with specified parameters, as a tensor completion method.\\
\noindent\textbf{HaLRTC: High accuracy low rank tensor completion \cite{liu2012tensor}}
is a tensor completion based on a certain definition of the trace norm for tensors together with convex optimization algorithms. Trace norm can be thought of as the tightest constraint for the rank of a tensor.

\begin{table*}
\rowcolors{2}{white}{black!05!white}
\renewcommand{\arraystretch}{1.1}
\renewcommand{\arraystretch}{1.1}
\centering
\resizebox{1\textwidth}{!}
{
\begin{tabular}{l ||c ||c c|| c c|| c c|| c c ||c c}
\hline
&& \multicolumn{2}{c}{\textbf{HaLRTC}} & \multicolumn{2}{c}{\textbf{CTRC}} & \multicolumn{2}{c}{\textbf{Cell}}  &
\multicolumn{2}{c}{\textbf{NTD-DR}}  &
\multicolumn{2}{c}{\textbf{CTTC} }
\\
Data & m. Rate 
&  $fit$~~ & RSE 
& $fit$ ~~& RSE 
& $fit$ ~~& RSE 
& $fit$ ~~& RSE 
& $fit$~~ & RSE \\
\hline
\hline
& 0.10
&0.80~~
&0.1980$\pm$ 0.001
& {0.97}~~
& {0.0401$\pm$0.002}
&0.92~~
&0.0817$\pm$ 0.003
&0.74~~
&0.2252$\pm$0.045
&\textbf{0.99~~}
&\textbf{0.0038$\pm$ 0.001}
\\
&0.20
&0.82~~
&0.1752$\pm$ 0.001
& {0.94}~~
& {0.0701$\pm$ 0.091}
& 0.89~~
& 0.1074$\pm$ 0.001
& 0.72~~
&0.2625$\pm$ 0.033
&\textbf{0.98~~}
&\textbf{0.0154$\pm$ 0.002}
\\
\textbf{DTD}
& 0.30
&0.62~~
&0.3795$\pm$ 0.005
& {0.93}~~
& {0.0798$\pm$ 0.064}
&0.86~~
&0.1334$\pm$0.001
&0.63~~
&0.3670$\pm$0.037
&\textbf{0.97~~}
&\textbf{0.0557$\pm$ 0.002}
\\
& 0.40 
&0.59~~
&0.4146$\pm$ 0.001
& {0.79}~~
& {0.2123$\pm$ 0.019}
&0.63~~
&0.3714$\pm$ 0.009
&0.62~~
&0.3784$\pm$0.029
&\textbf{0.93~~}
&\textbf{ 0.0711$\pm$ 0.002}
\\
& 0.50
&0.30~~
&0.7071 $\pm$ 0.002
& {0.61}~~
& {0.4002$\pm$ 0.073} 
&0.57~~
&0.4314$\pm$ 0.012
&0.53~~
&0.4694$\pm$0.035
& \textbf{0.95~~}
&\textbf{0.0509$\pm$ 0.031}
\\
\hline
\hline
\hline
%LINCS
& 0.10
&0.69~~
&0.3109$\pm$0.066
& {0.90}~~
& {0.1078$\pm$ 0.003}  
&0.86~~
&0.1395$\pm$0.003
&0.89~~
& 0.1092$\pm$ 0.032
&\textbf{0.99}~~
&\textbf{{0.0080$\pm$0.001}}
\\
& 0.20
&0.57~~
&0.4244$\pm$0.003
& 0.90~~
&   {0.1391$\pm$ 0.002} 
&0.84~~
&0.1541$\pm$0.009
&0.86~~
& 0.1403$\pm$ 0.002
&\textbf{0.99}~~
& \textbf{{0.0113$\pm$0.001}}
\\
\textbf{LINCS}
& 0.30
&0.46~~
&0.5381$\pm$0.007
& 0.89~~
& 0.1936$\pm$ 0.020
&0.81~~
&0.1901$\pm$0.006
&0.83~~
& 0.1728$\pm$ 0.004
&\textbf{0.98}~~
&\textbf{{0.0138$\pm$0.002}}
\\
& 0.40
&0.35~~
&0.6487$\pm$0.001
& 0.84~~
& 0.2590$\pm$0.001
&0.77~~
&0.2269$\pm$0.003
&0.80~~
& 0.2015$\pm$ 0.005
&\textbf{0.98}~~
&\textbf{{0.0160$\pm$0.001}}
\\
& 0.50
&0.02~~
&0.9749$\pm$0.004
& 0.79~~
& 0.2894$\pm$0.003
&0.52~~
&0.4764$\pm$0.004
&0.78~~
& 0.2245$\pm$ 0.005
&\textbf{0.98}~~
&\textbf{{0.0179$\pm$0.002}}
\\
\hline
\hline
\end{tabular}
}
\vspace{0.09in}
\caption{%{\color{blue} 
The comparison of the performance of four tensor completion methods against CTTC over two datasets, LINCS and DTD \cite{jamali2022ntd}. Here, m.Rate denotes the missing rate, i.e. the percentage of the entries that has been held-off. The results are the average over 10 experiments and the standard deviation is shown by the number followed by $\pm$. The metrics used here are Relative Squared Error  (RSE) and the model $fit$. }
%}
\label{tab:res}
\end{table*} 

\begin{table*}
\renewcommand{\arraystretch}{1.4}
\renewcommand{\arraystretch}{1.4}
\centering
\resizebox{0.7\textwidth}{!}
{
\begin{tabular}{||c|| c ||c c c c c||}
\hline
Data
&metric
&{~~\textbf{HaLRTC}~~}
&{~~\textbf{CTRC}~~} 
& {~~\textbf{Cell}~~}  &{~~\textbf{{NTD-DR}}~~} 
&{~~\textbf{CTTC}~~ }
\\
\hline
\hline
{\textbf{DTD}}
&\cellcolor{black!05!white}
run time (s) 
& \cellcolor{black!05!white} 
3.75$\pm$ 0.90
&\cellcolor{black!05!white}
48.91$\pm$ 4.12
&\cellcolor{black!05!white}
28.83 $\pm$ 2.08
& \cellcolor{black!05!white}
4800$\pm$180.0
& \cellcolor{black!05!white}
10.30 $\pm$ 1.11
\\
\hline
\hline
{\textbf{LINCS}}
&\cellcolor{black!05!white}
run time (s)
& \cellcolor{black!05!white}
50.99$\pm$ 1.94
& \cellcolor{black!05!white}
61.01$\pm$ 7.32
&\cellcolor{black!05!white}
70.74$\pm$0.33
&\cellcolor{black!05!white}
5600$\pm$126.0
& \cellcolor{black!05!white}
{{971.55$\pm$60.16}}
\\
\hline
\hline
\end{tabular}}
\vspace{0.09in}
\caption{
The average of the run time for 10 iterations for each method. Note that the HaLRTC and Cell cannot incorporate side information, while the other three can.}
%}
\label{tab:time}
\end{table*}

\begin{table*}[ht]
%\rowcolors{2}{white}{black!05!white}
\renewcommand{\arraystretch}{0.7}
\renewcommand{\arraystretch}{0.7}
\centering
\resizebox{0.4\textwidth}{!}
{
\begin{tabular}{c|| c |c |c c c }

%\rowcolors{1}{*}{white}{black!05!white}
%\rowcolor{black!05!white}
\textbf{m. Rate}& \textbf{metric}&{~~Drugs~~} &{~~Genes~~} &{~~Cell-lines~~} 
\\
\hline
\hline
\multirow{2}{*}{\textbf{0.1}}
&\cellcolor{black!05!white}
RSE
& \cellcolor{black!05!white}
0.103
& \cellcolor{black!05!white}

{0.010}
& \cellcolor{black!05!white}
{0.0974}

\\
&
$fit$
&{0.9}
& {0.99}
& {0.90}
\\
\hline
\hline
\multirow{2}{*}{\textbf{0.2}}
&\cellcolor{black!05!white}
RSE
& \cellcolor{black!05!white} 
{0.145}
&\cellcolor{black!05!white}

{0.012}
&\cellcolor{black!05!white}
{0.164}
\\
&$fit$
& {0.85}
&{0.99}
&{0.84}
\\
\hline
\hline
\multirow{2}{*}{\textbf{0.3}}
&\cellcolor{black!05!white}
RSE
& \cellcolor{black!05!white}
{0.175}
& \cellcolor{black!05!white}
{0.0144}
& \cellcolor{black!05!white}
{0.172}
\\
&
$fit$
&{0.82}
& {0.99}
& {0.83}
\\
\hline
\hline
\end{tabular}}
\vspace{0.09in}
\caption{
Here, using the LINCS dataset, the missing values are chosen only from the set of drugs or sets of genes as opposed to being randomly selected.}
%}
\label{tab:cs}
\end{table*}

\section{Experimental Results}\label{sec:ER}

The related codes for CTTC are developed in MATLAB with partial use of previously developed methods \cite{bagherian2022bilevel, bagherian2024tensor, bagherian2021coupled, chen2013simultaneous}. To evaluate the performance of the proposed
CTTC method, we ran it on the two datasets discussed in Section~\ref{sec:data}.  To this end, we considered multidimensional arrays containing auxiliary information between the entries of each mode of the tensor. We used two metrics to compare the results:  Relative Squared Error  (RSE) and $fit$. Fit is calculated as $1-{\| \T X- \wh{\T X}\|}/{\|\T X\|}$, where $\T X$ is the initial tensor and $\wh{\T X}$ is the recovered tensor by CTTC.
In Table~\ref{tab:res}, ``m. Rate'' denotes the missing rate, which is the percentage of the data being held out. The results reported in Table~\ref{tab:res} are the average over $10$ experiments along with $\pm$ the standard deviation (SD). Overall, CTTC outperforms the other approaches across the whole range of missing rates (Table 1). The second-best method is CTRC. CTTC especially shines on more sparse tensors, likely because it better incorporates the tensor-valued side information.

We have also compared the methods in terms of average run-time and the number of iterations required for the method to converge. The results are reported in Table~\ref{tab:time}. The maximum number of iteration for both CTTC and HaLRTC is set to be $30$. However, if the convergence threshold is met, the algorithms may terminate earlier. Even though CTTC method involves more computation due to additional arrays integrated to the algorithm, as shown in Table~\ref{tab:time}, the method still converges with a low number of iterations and over relatively small run-time.  The algorithm used in Cell has a fixed number of iterations set to 200, yet, the method tends to converge with much fewer iterations. NTD-DR uses 10 fold cross validation and reports the metrics for each fold. The method takes a long time to complete, as reported in Table~\ref{tab:time}. For the CTRC method, the maximum iteration is set to be 100, unless the method reaches the termination threshold, faster. CTTC, CTRC, and HaLRTC are written in MATLAB. The ''Cell'' method is a slight modification of the FaLRTC method \cite{liu2012tensor}, where the authors have used customized $\alpha$ values. This method, too, is developed in MATLAB and uses both maximum number of iteration and threshold. The method NTD-DR is developed in Python\cite{jamali2022ntd}. Since the methods are coded in different programming languages and strategies for determining the number of iterations, the run-time comparisons should be taken with a grain of salt.  Initially, the authors split the testing and training dataset with ratio $0.05/99.95$ and reported the results.

The regularization parameters $0\le \varrho_\ell\le 1$, $\ell=1,2,3$, in Eq.~\eqref{eq:main} are optimized based on the performance of the
algorithm during the execution of CTTC methods. They are highly correlated with the dataset being used both for the main tensor input and for the arrays as auxiliary information. For NTD-DR and LINCS datasets, the triple regularization parameters are set to be $(0.01, 0.01, 0.1)$ and $(0.1,0.1,0.1)$, respectively. Setting either of the parameters $\varrho_\ell$ to zero
simply ignores the role of one of the modes, in case no side information is available for such mode. In Table~\ref{tab:res}, we report the performance evaluation for the CTTC method by adding side information for each mode gradually. It is meant to show the effect of utilizing additional information over the overall performance of the method. For instance, in one of experiments over NTD-DR dataset, for a fixed missing rate as $0.2$, the overall method RSE=$0.0654$. Putting $\varrho_\ell=0$, for one $\ell$ at a time results in higher RSE. Putting all three $\varrho_\ell=0$, meaning no side information is used, resulted in RSE=$0.4190$ and $fit=0.69$. While mainly depending on the nature of the database in use, empirically, smaller values of these regularization parameters lead to a better prediction performance. 

\begin{table*}
%\rowcolors{2}{white}{black!05!white}
\renewcommand{\arraystretch}{1.4}
\renewcommand{\arraystretch}{1.4}
\centering
\resizebox{0.8\textwidth}{!}
{
\begin{tabular}{||c|| c ||c ||c ||c ||c ||c }
\hline
%\rowcolors{1}{*}{white}{black!05!white}
%\rowcolor{black!05!white}
\textbf{m. Rate}& \textbf{metric}&{~~${\varrho_1=0,\varrho_2=0,\varrho_3=0}$~~} &{~~\textbf{$\varrho_1\neq0,~\varrho_2=0,~\varrho_3=0$}~~} & {~~\textbf{$\varrho_1\neq0,~\varrho_2\neq0,~
\varrho_3=0$}~~}  &
{~~\textbf{$\varrho_1\neq0,~\varrho_2\neq0,~\varrho_3\neq0$}~~}
\\
\hline
\hline
\multirow{2}{*}{\textbf{0.1}}
&\cellcolor{black!05!white}
RSE
& \cellcolor{black!05!white}
0.2233
& \cellcolor{black!05!white}
0.0999
&\cellcolor{black!05!white}
0.0690

& \cellcolor{black!05!white}
0.0038
\\
&
$fit$
&0.78
& 0.90
& 0.93
& 0.99
\\
\hline
\hline
\multirow{2}{*}{\textbf{0.2}}
&\cellcolor{black!05!white}
RSE
& \cellcolor{black!05!white} 
0.4190
&\cellcolor{black!05!white}
0.1900
&\cellcolor{black!05!white}
0.1080
& \cellcolor{black!05!white}
0.0154
\\
&$fit$
& 0.69
&0.81
&0.89
& 0.98
\\
\hline
\hline
\end{tabular}}
\vspace{0.09in}
\caption{%{\color{blue}
The effect of adding side information on the performance evaluation of the CTTC method, considering different missing rates over DTD dataset.  }
%}
\label{tab:ell}
\end{table*} 

The CTTC method is flexible and can perform under an arbitrary number of modes for the coupled arrays as side information. This allows the method to freely incorporate additional information, when available, using different data sources. The method itself can be used as a pre-processing method if the side arrays are under-observed. It is also possible to synthetically create side tensors, using pair-wise distances, for instance, when applicable. This approach is more suitable for image data, where in fact the pair-wise distances carry certain (additional) geometrical information  \cite{bagherian2022bilevel}. 

In the case of image data, each pixel in an image can be represented as a vector consisting of RGB values, thus allowing pair-wise distance calculations that are reflective upon the similarity of the color in each pixel. However, in our case we do not have access to such a numerical representation for any mode while not utilizing the similarity matrices. Taking our three mode tensor, each entry in the tensor corresponds to a specific interaction of drug, cell line, and gene, that is binary. The labels of each mode are represented by a non-numeric string. Thus, we have no additional numerical vector other than the position of the value in the tensor and are unable to gain more information to create artificial side tensors while only using the target tensor. 

\subsection{Case Study}
The LINCS dataset allows modeling and prediction of whole cell lines using drug and gene information. Instead of considering a missing rate, which randomly marks a percentage of the data as unknown, we tend to predict missing entries being held out from each set of input, separately. In Table ~\ref{tab:cs}, we reported the RSE and $fit$ when the missing values are either from the set of drugs or genes at different percentages. The results are comparable to those reported in Table~\ref{tab:res} for the same missing rate. This indicates that CTTC can be used to predict drug-induced gene expression even for completely unseen drugs or genes--a very useful capability for drug repurposing. However, we observed that due to the sparsity of cell-line side information, unseen cell line prediction does not lead to meaningful results.

\subsection{Ablation Study}
To assess the individual contributions of distance metric learning and tensor completion, we performed an ablation study by evaluating the impact of each mode-specific metric tensor. Each metric tensor is multiplied by its corresponding regularization parameter $\varrho_\ell$. To isolate the effect of each learned metric, we systematically "switched off" the influence of the metrics by setting their corresponding regularization parameters $\varrho_\ell = 0$ for each mode $\ell = 1, 2, 3$. This configuration corresponds to a tensor completion method that does not use any learned distance functions for the modes. We then reintroduced each metric tensor one at a time and observed the performance improvements. The results of this analysis are presented in Table~\ref{tab:ell}, where we show how progressively incorporating each metric tensor leads to enhanced performance. This ablation study highlights the significant contribution of the learned distance metrics to the overall success of the proposed method. According to the Table~\ref{tab:ell}, if no side information is used, i.e. $\varrho_i=0$ for $i=1,2,3$ the metric involved in the model is simply the Euclidean metric, and the average error of the method is relatively high ($0.2233$ and $0.4190$ for $0.1$ and $0.2$ missing rates, respectively). However, by adding layers of side information in each mode, step by step, the performance of the CTTC model improves drastically. Specifically, adding the information about the drugs $\varrho_1\neq0$, while ignoring the other side information, $\varrho_2=\varrho_3=0$, improves the overall results by updating the distance metric function which is responsible for calculating similarity information about the set of drugs.

\subsection{Time, Complexity, and Sensitivity Analysis}

Considering the core term in the  objective function given in Eq.~\eqref{eq:main}, the time complexity for the tensor-matrix multiplications for a tensor \(\T Z\) of size \(n_1 \times n_2 \times n_3\) and matrix \(\m U^{(\ell)}\in \SL_{n_\ell}\) of size \(n_{\ell} \times n_{\ell}\) is \(\mathcal O(n_{\ell} n_{\ell} n_{- \ell})\), where \(n_{- \ell}\) represents the sizes of the other two dimensions. The overall time complexity for performing the three-mode multiplication and calculating the Mahalanabois norm is \(\mathcal O_1(n_1 n_2 n_3 (n_1 + n_2 + n_3)).\)
Each regularization term \(J_\ell(\m U^{(\ell)})\), on the other hand, involves matrix norms which generally have a time complexity of \(\mathcal O(n_\ell^2)\) for each matrix \(\m U^{(\ell)}\). Therefore, the total time complexity for the regularization terms is \(\mathcal O_2(\sum\limits_{\ell=1}^3 n_\ell^2 ).\) In the distance metric regularization term, the matrix multiplication and trace computation for each term have a time complexity of \(\mathcal O(n_\ell^2)\) for each \(j\), and summing over \(\text{num}_\ell\) gives a complexity of \(O(\text{num}_\ell n_\ell^2)\) for each \(\ell\). Thus, the time complexity for the distance metric regularization is \(
   \mathcal O_3(\text{num}_1 n_1^2 + \text{num}_2 n_2^2 + \text{num}_3 n_3^2).
   \)
Combining the complexities from the tensor operations, regularization terms, and distance metric regularization, the total time complexity is approximately:
\begin{equation}
  \mathcal O\bigg (n_1 n_2 n_3 \sum\limits_{\ell=1}^3 n_\ell  + \sum\limits_{\ell=1}^3 n_\ell^2 \big[ 1 + \text{num}_\ell n_\ell^2 \big ]\bigg).   \nonumber
\end{equation}
In order to analyse the space complexity one may note that the the space required for storing the tensor \(\T Z\) is \(\mathcal O(n_1 n_2 n_3)\) and each matrix \(\m U^{(\ell)}\), requires \(\mathcal O(n^2_\ell)\) space. Therefore, the total space for the matrices is \(\mathcal O(\sum\limits_{\ell=1}^3 n_\ell^2)\).\\
Considering the distance metric and regularization terms, the matricization of tensors \(\m S^{(\ell)}\) along the $j$th mode, i.e., \(\m S^{(\ell)}_{(j)}\) and \(\m D_\ell\) would add \(\mathcal O(\text{num}_\ell n_\ell^2)\) space for each \(\ell=1, 2,3\), depending on how the matrices are stored.

Thus, the total space complexity is approximately
\begin{equation}
 \mathcal O\bigg (n_1 n_2 n_3\sum\limits_{\ell=1}^3  n^2_\ell\big[1+  \text{num}_\ell\big]\bigg).  \nonumber 
\end{equation}
As expected, the ``dominant factor'' in both time and space complexity is the size of the given tensor, i.e., \(n_1 n_2 n_3\) along with any additional scaling factors from the regularization terms or auxiliary data.

The optimal regularization parameters $\varrho_\ell$ for $\ell = 1, 2, 3$ in Eq.~\eqref{eq:main} are selected based on the similarity tensors for each mode, and they are influenced by the sparsity and reliability of the auxiliary information. To evaluate the sensitivity of the proposed method to variations in the regularization parameters $\varrho_\ell$, and to better understand the role of each parameter, we first explore the effect of setting each $\varrho_\ell$ (for $\ell = 1, 2, 3$) to zero. This approach effectively removes the influence of the corresponding metric tensor for the respective mode, allowing us to optimize the remaining parameters independently based on their individual performance. This strategy is feasible because the metric tensors under the distance metric learning framework are treated as independent for each mode. The other parameters involve in the optimization, depending on the iterative process, are optimized through a grid search.

\subsection{Application of CTTC in Drug Repurposing}\label{sec:application}

To further analyze the performance of CTTC in the drug repurposing task, we have conducted a stratified evaluation to control for data heterogeneity and to ensure that the model’s performance reflects generalizability across biologically and chemically relevant subgroups. Stratification, in this context, refers to the systematic partitioning of a dataset into distinct, non-overlapping subgroups (strata), each defined based on specific, predefined characteristics pertinent to biological function, chemical structure, or clinical relevance.

Formally, let the dataset be defined as
\(
D = \{(x_i, y_i)\}_{i=1}^{n},
\)
where \( x_i \in \mathbb{R}^d \) denotes the feature vector associated with the \(i^{\text{th}}\) sample—such as molecular descriptors, drug-target profiles, or gene expression signatures—and \( y_i \in \mathcal{Y} \) represents the corresponding outcome, such as therapeutic response, toxicity level, or predicted interaction score. Stratification entails dividing the dataset \(D\) into \(k\) disjoint subsets, i.e.,
\(
D = \bigcup_{j=1}^{k} D_j \quad \text{with} \quad D_j \cap D_{j'} = \emptyset \quad \text{for } j \ne j',
\)
such that each subset \(D_j\) contains samples that share one or more defining attributes \(A_j\), i.e.,
\(
\forall (x_i, y_i) \in D_j, \quad x_i \in A_j.
\)
This process enables controlled analyses by ensuring that stratifying variables—such as disease subtype, molecular pathway membership, chemical scaffold, or gene family—are appropriately represented within each group. This methodological design supports the evaluation of the model’s consistency, robustness, and capacity to generalize across varied biomedical subdomains, thereby improving the interpretability and trustworthiness of drug discovery outcomes.

In our experimental analysis, we applied this stratification framework to the LINCS dataset while maintaining a fixed missing rate of \(0.2\). To assess performance variability across biologically and chemically meaningful strata, we stratified the data along two axes: (1) 12 distinct chemical structural classes and (2) 520 gene families. The gene family classifications are based on the standardized definitions provided by the HUGO Gene Nomenclature Committee (HGNC)~\cite{hgnc2022gene}. Under the specified missing data condition, CTTC achieved an overall Relative Squared Error  (RSE) of \(0.0126\), indicating high predictive accuracy.

To provide a more granular assessment, we report the model’s performance (RSE) separately for each of the structural classes and gene families. The stratified results are provided in Table~\ref{tab:str}, while a visual summary highlighting gene families with five or more members is presented in the bar plot shown in Figure~\ref{fig:rse_gene_families}. This visualization helps identify consistent patterns of performance across diverse gene groups, shedding light on the biological contexts in which CTTC’s predictions are most reliable. Notably, in Figure~\ref{fig:rse_gene_families}, several gene families—such as those involved in signal transduction, kinase activity, or transcription regulation—demonstrate relatively low RSE values, indicating reliable prediction performance within these biologically significant domains. Conversely, a small number of families exhibit elevated RSE values, which may be attributed to higher functional heterogeneity or limited data density within those subgroups.

Table~\ref{tab:str} presents the group-wise Relative Squared Error (RSE) values obtained by the CTTC model across 12 distinct chemical structural classes from the LINCS dataset. Each group corresponds to a specific chemical functional group or scaffold—such as amines, esters, alkynes, and aromatics—which are fundamental to molecular reactivity and biological activity. As observed, the \textit{Alkyne} class exhibits the lowest RSE at $0.0107$, suggesting that CTTC achieves particularly high fidelity in predicting interactions involving alkyne-containing compounds. Similarly, the model performs well on the \textit{Alcohol}, \textit{Sulfonamide}, and \textit{Ether} groups, each achieving RSEs close to or below $0.0115$. In contrast, the \textit{Ester}, \textit{Alkene}, and \textit{Carboxylic acid} groups yield comparatively higher RSEs, ranging from $0.0133$ to $0.0136$, indicating moderate challenges in learning from these structural patterns. Overall, the stratified RSE analysis demonstrates CTTC’s robustness and adaptability across chemically diverse input spaces, supporting its potential utility in broad-spectrum drug discovery tasks.

 In another experiments over DTD dataset where 30\% of the data are being held out from an entire row. We performed stratification based on disease categories using MeSH terms and stratified targets using Gene Ontology (GO) annotations. The overall RSE is reported as $0.0877$, which is slightly higher than the value reported in Table~\ref{tab:res} for the same missing rate. Stratified analysis revealed that the algorithm's performance varies across specific subgroups, indicating heterogeneous effectiveness in different biological contexts. The accompanying histogram aggregates RSE values for all GO terms with at least five associated targets. Similar trends were observed when analyzing performance across disease categories, with RSE values ranging from 0 to $0.25$. Notably, RSE values of 0 may correspond to cases with no masked entries for the respective class.

\begin{figure*}[!t]
  \centering
  \includegraphics[width=\textwidth]{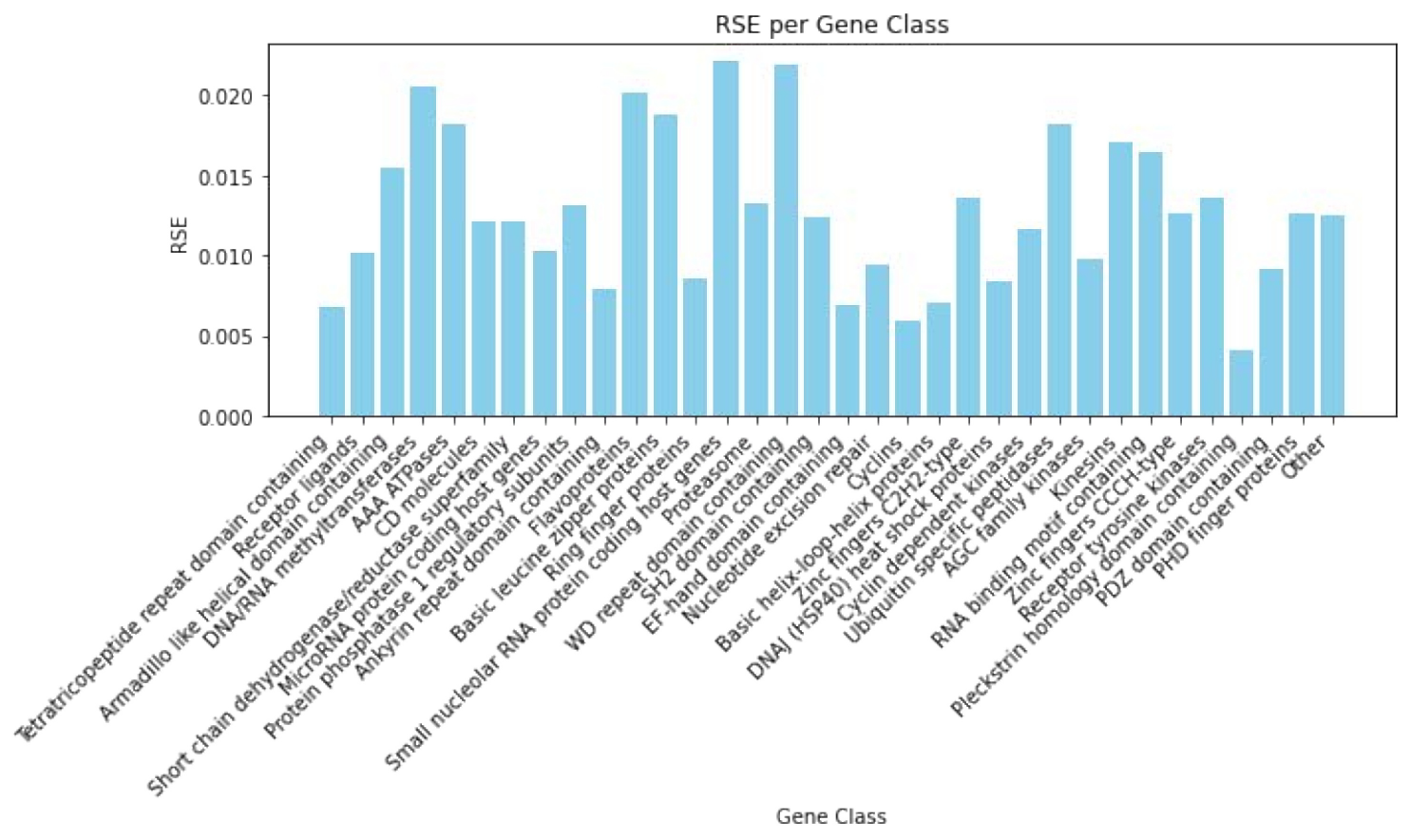}
     \caption{RSE per gene family with at least $5$ members.}
     \label{fig:rse_gene_families}
\end{figure*}

\begin{table*}[ht]
\renewcommand{\arraystretch}{1.6}
\centering
\resizebox{\textwidth}{!}{
\begin{tabular}{|c||c|c|c|c|c|c|c|c|c|c|c|c|}
\hline
\hline
\textbf{Group} & Amine & Ester & Alkyne & Ether & Carboxylic acid & Aromatic & Alcohol & Sulfonamide & Amide & Alkene & Ketone & Phenol \\
\hline
\textbf{RSE} & $0.0131$ & $0.0135$ & $0.0107$ & $0.0115$ & $0.0133 $& $0.0126$ & $0.0114$ & $0.0114$ & $0.0120$ & $0.0136$ & $0.0124$ & $0.0122 $\\
\hline
\hline 
\end{tabular}}
\vspace{0.1in}
\caption{Group-wise Relative Squared Error (RSE) values across $12$ chemical structural classes on the LINCS dataset.}
\label{tab:str}
\end{table*}

\section{Conclusion}\label{sec:conc}
In this manuscript, we introduced a tensor completion method, coined CTTC, and evaluated its performance against four state-of-the-art tensor completion methods on two benchmark datasets.

The CTTC method is highly generalizable and can be applied to a wide range of tasks where data is sparse and contains missing entries. The core strength of CTTC lies in its ability to predict missing values while simultaneously learning an appropriate distance metric to define similarity relationships between entities, such as drugs, targets, or cell lines, through distance metric learning. This dual capability allows for more accurate predictions and better understanding of the underlying data structure.

While we have demonstrated its effectiveness in specific applications such as Drug-Target Interaction (DTI) prediction and cell-line perturbation prediction, the method is not limited to these domains. CTTC is particularly useful in scenarios where multi-relational data—often with missing values—needs to be completed, and auxiliary information can enhance predictive accuracy. The method's flexibility comes from its ability to integrate multiple sources of information, such as structural data, biological features, or external datasets, to guide the completion process. This makes CTTC applicable to a wide array of domains, including but not limited to, drug discovery, biomarker identification, disease modeling, and personalized medicine. Moreover, CTTC can be extended to any task that involves multi-dimensional data with missing entries, where learning from auxiliary information (e.g., disease categories, target types, or compound structural features) can improve the quality of predictions. Its versatility and adaptability to various types of incomplete data make it a powerful tool for real-world applications across diverse fields.

Building upon the proposed framework, an interesting direction for future research is the development of a Riemannian manifold formulation of CTTC, in which the optimization is performed directly on suitable fixed-rank tensor manifolds. Such a formulation would enable the investigation of intrinsic geometric properties of the optimization problem, including geodesic convexity, sectional curvature, and their influence on the optimization landscape and convergence behavior.

\section*{Data Statement \& Supplementary Data}\label{supp}

The following publicly available datasets are used for experimental purposes in thismanuscript: 
\begin{enumerate}[label=(\roman*)]
	\item Library of Integrated Network-based Cellular Signatures (LINCS) dataset {\url{https://lincs.hms.harvard.edu/db/}} 
	\item ChemicalChecker \url{https://pubmed.ncbi.nlm.nih.gov/32440005/} \cite{duran2020extending}
	\item DrugBank \url{https://academic.oup.com/nar/article/34/suppl_1/D668/1132926} \cite{wishart2006drugbank}
	\item  NTD-DR dataset  \url{https://journals.plos.org/plosone/article?id=10.1371/journal.pone.0270852} \cite{jamali2022ntd}
	\item STRING \url{https://pubmed.ncbi.nlm.nih.gov/33237311/}  \cite{szklarczyk2021string}
	\item  RegNetwork \url{https://academic.oup.com/nar/article/54/D1/D1234/8234001} \cite{liu2015regnetwork}
	
\end{enumerate}

\noindent A subset of the processed data used to assess the proposed method, together with the associated code and usage instructions, will be made publicly available in the following GitHub repository upon publication of this paper: \href{https://github.com/mbagherian/CTTC}{https://github.com/mbagherian/CTTC}.

%%%%%%%%%%%%%%%%%%%%%%%%%%%%%%%%%%%%%%%%%%%%%%%%%%%%%%%%%%%%%%%%%%%%%%%%%%%%%%%%%%%%%

%%%%%%%%%%%%%%%%%%%%%%%%%%%%%%%%%%%%%%%%%%%%%%%%%%%%%%%%%%%%%%%%%%%%%%%%%%%%%%%%%%%%%%%%%
%%%%%%%%%%%%%%%%%%%%%%%%%%%%%%%%%%%%%%%%%%%%%%%%%%%%%%%%%%%%%%%%%%%%%
% \bibliographystyle{ieeetr}
 \bibliographystyle{splncs04}
\bibliography{REF_new}

\end{document}